\documentclass[11pt, a4paper]{amsart}

\usepackage[T1]{fontenc}

\usepackage{a4wide}
\usepackage{graphicx,enumitem}
\usepackage{subfigure}
\usepackage{units}
\usepackage[usenames,dvipsnames,table]{xcolor}
\usepackage[sectionbib,numbers,sort&compress,merge]{natbib}

\usepackage{framed}

\usepackage{parskip}

\usepackage{ulem}
\usepackage{amsmath}
\usepackage{amssymb}
\usepackage{latexsym}

\usepackage[luatex,
            pdftitle={Isotropic Q-fractional Brownian motion on the sphere},
            pdfauthor={A. Lang and B. Müller},
            bookmarksopen,
            colorlinks=true,
            linkcolor=black,
            urlcolor=blue,
	   		citecolor=black
]{hyperref}

\usepackage{dsfont}

\usepackage{tikz}

\allowdisplaybreaks
\usepackage{mathtools}

\newcommand{\IP}{\mathbb{P}}
\newcommand{\R}{\mathbb{R}}

\newcommand{\N}{\mathbb{N}}

\newcommand{\IS}{\mathbb{S}}
\newcommand{\IT}{\mathbb{T}}

\newcommand{\ga}{\alpha}
\newcommand{\gb}{\beta}

\newcommand{\gk}{\kappa}
\newcommand{\gl}{\lambda}

\newcommand{\gL}{\Lambda}
\newcommand{\gO}{\Omega}

\newcommand{\cA}{\mathcal{A}}

\newcommand{\cN}{\mathcal{N}}
\newcommand{\cO}{\mathcal{O}}

\DeclareMathOperator{\E}{\mathbb{E}} 
\DeclareMathOperator{\Var}{\mathsf{Var}} 

\DeclareMathOperator{\trace}{Tr}
\DeclareMathOperator{\Id}{Id}

\let\Re\relax
\DeclareMathOperator{\Re}{Re}
\let\Im\relax
\DeclareMathOperator{\Im}{Im}

\newcommand{\KL}{Karhunen--Lo\`eve }

\newtheorem{lemma}{Lemma}[section]

\newtheorem{theorem}[lemma]{Theorem}
\newtheorem{corollary}[lemma]{Corollary}

\theoremstyle{remark}
\newtheorem{remark}[lemma]{Remark}

\theoremstyle{definition}
\newtheorem{definition}[lemma]{Definition}
\newtheorem{assumption}[lemma]{Assumption}

\begin{document}

\title[Isotropic \texorpdfstring{$Q$}{Q}-fractional Brownian motion on the sphere]{
Isotropic \texorpdfstring{$Q$}{Q}-fractional Brownian motion on the sphere: regularity and fast simulation}

\author[A.~Lang]{Annika Lang} \address[Annika Lang]{\newline Department of Mathematical Sciences
	\newline Chalmers University of Technology \& University of Gothenburg
	\newline S--412 96 G\"oteborg, Sweden.} \email[]{annika.lang@chalmers.se}

\author[B.~M\"uller]{Bj\"orn M\"uller} \address[Bj\"orn M\"uller]{\newline Department of Mathematical Sciences
	\newline Chalmers University of Technology \& University of Gothenburg
	\newline S--412 96 G\"oteborg, Sweden.} \email[]{bjornmul@chalmers.se}

\thanks{Acknowledgment. The authors would like to thank Stig Larsson and Ruben Seyer for fruitful discussions and the aonymous referees for helpful comments.
This work was supported in part by the European Union (ERC, StochMan, 101088589), by the Swedish Research Council (VR) through grant no.\ 2020-04170, by the Wallenberg AI, Autonomous Systems and Software Program (WASP) funded by the Knut and Alice Wallenberg Foundation, and by the Chalmers AI Research Centre (CHAIR).
Views and opinions expressed are however those of the author(s) only and
do not necessarily reflect those of the European Union or the European Research Council Executive Agency.
Neither the European Union nor the granting authority can be held responsible for them.}


\begin{abstract}
As an extension of isotropic Gaussian random fields and $Q$-Wiener processes on $d$-dimensional spheres, isotropic $Q$-fractional Brownian motion is introduced and sample H\"older regularity in space-time is shown depending on the regularity of the spatial covariance operator~$Q$ and the Hurst parameter~$H$. The processes are approximated by a spectral method in space for which strong and almost sure convergence are shown. The underlying sample paths of fractional Brownian motion are simulated by circulant embedding or conditionalized random midpoint displacement. Temporal accuracy and computational complexity are numerically tested, the latter matching the complexity of simulating a $Q$-Wiener process if allowing for a temporal error.
\end{abstract}

\maketitle

\section{Introduction}\label{sec:intro}

The approximation of stochastic partial differential equations (SPDEs) and corresponding error analysis have been performed for the last 25~years to efficiently compute solutions to models under uncertainty. In most models, the equations are driven by Wiener processes, which yield SPDE solutions with H\"older regularity in time limited by~1/2. One option to get more flexible smoothness in time is to consider infinite-dimensional fractional Brownian motions. Theoretical results on the properties of solutions are available in Euclidean space and in abstract Hilbert and Banach spaces, see \cite{issoglioCylindricalFractionalBrownian2014} for an overview. At the same time, analysis and numerical approximations on non-Euclidean domains are still rare. Motivated by applications in environmental modeling and astrophysics, first analysis and approximations for fractional equations on the sphere have been considered in~\cite{ABOW18, LX19}. An overview over space-time models in Euclidean space and on the sphere is given in~\cite{porcu30YearsSpace2021}.

In this work, we take a step back to carefully analyze and efficiently simulate fractional Brownian motion on spheres in any dimension as an important building block and input for the later simulation of SPDEs. In the first part, we construct isotropic $Q$-fractional Brownian motion with varying space regularity described by the covariance operator~$Q$ based on the Hilbert-space framework of~\cite{greckschQfractionalBrownianMotion2009,duncanFractionalBrownianMotion2002}, and the theory of isotropic Gaussian random fields on spheres developed in~\cite{MP11,langIsotropicGaussianRandom2015}. We show the existence of a continuous modification with optimal H\"older regularity in space, depending on~$Q$, and in time, bounded by the Hurst parameter~$H \in (0,1)$.

In the second part, we approximate $Q$-fractional Brownian motion by a spectral method in space and show strong and almost sure convergence with rates determined by the smoothing properties of~$Q$.
The temporal behavior is then determined by independent sample paths of real-valued fractional Brownian motion.
For their simulation, we exploit an exact method using circulant embedding and fast Fourier transforms with computational complexity~$\cO(N \log N)$ in the number of time steps~$N$ (cf., e.g.,~\cite{perrinFastExactSynthesis2002} and references therein) and compare it to an approximate method, called conditionalized random midpoint displacement, of computational complexity~$\cO(N)$~\cite{norrosSimulationFractionalBrownian1999}. The latter allows to generate the correlated increments of fractional Brownian motion with the same asymptotic speed as the independent increments of Brownian motion. Therefore, we achieve the same complexity for the simulation of $Q$-fractional Brownian motion as for $Q$-Wiener processes on the sphere. We compare their performance also with respect to the constants in the $\cO$-notation and numerically show the decay rate of the error of the midpoint displacement method. In Figure~\ref{fig:sample_QfBm}, we show sample paths generated with the same noise for Hurst parameter $H=0.1,0.5,0.9$ at times $T= 1, 2, 3$. We observe that while the spatial regularity is similar, the temporal behavior depends on $H$. For $H=0.1$, the correlation between temporal increments is negative, so the process stays close to $0$ for large $T$. For $H=0.9$, this correlation is positive, so we observe a consistent temporal trend. In the middle row, $H=0.5$ is the standard $Q$-Wiener process with independent increments.

\begin{figure}[tb]
	\centering
	\subfigure[$H=0.1$, $T=1$.]{
		\includegraphics[width=0.3\textwidth,trim={30em 15em 30em 15em},clip]{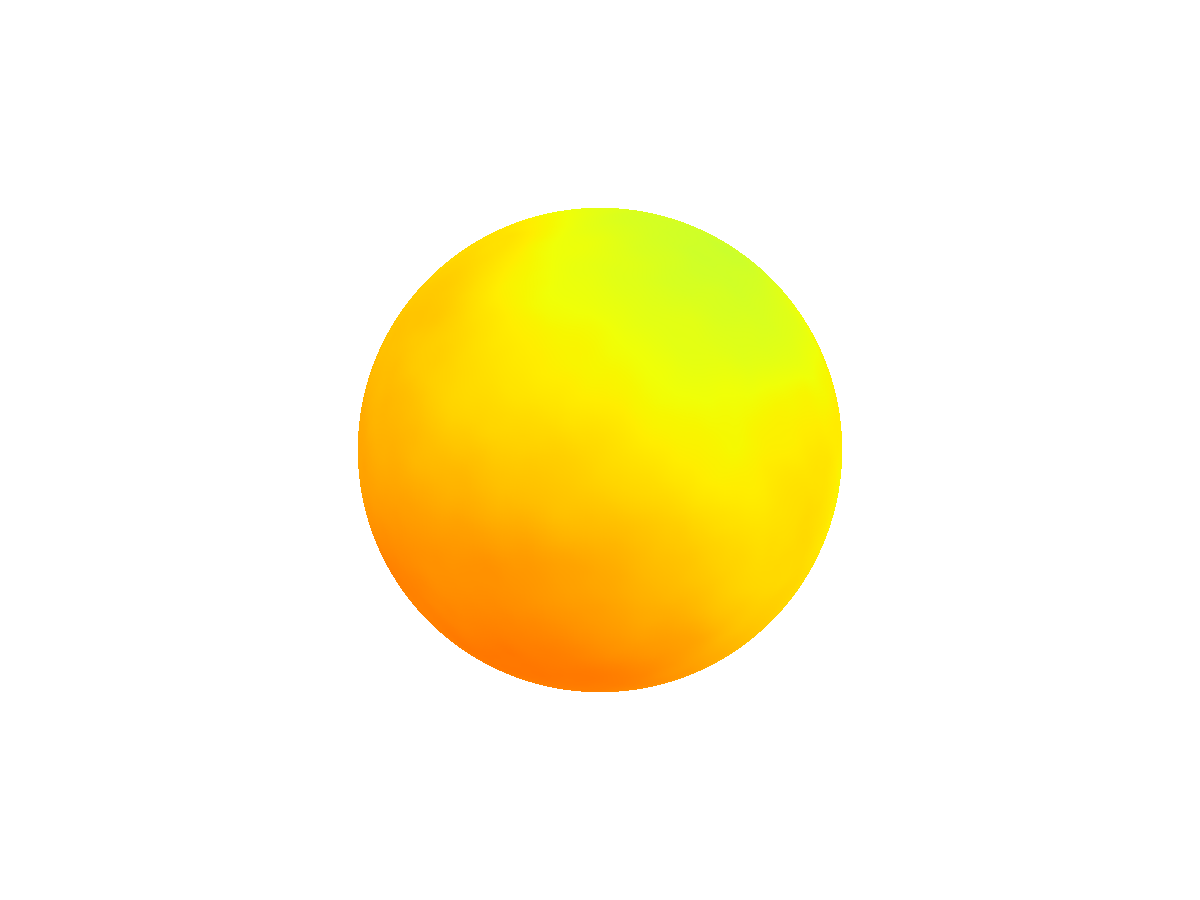}}
	\subfigure[$H=0.1$, $T=2$.]{
		\includegraphics[width=0.3\textwidth,trim={30em 15em 30em 15em},clip]{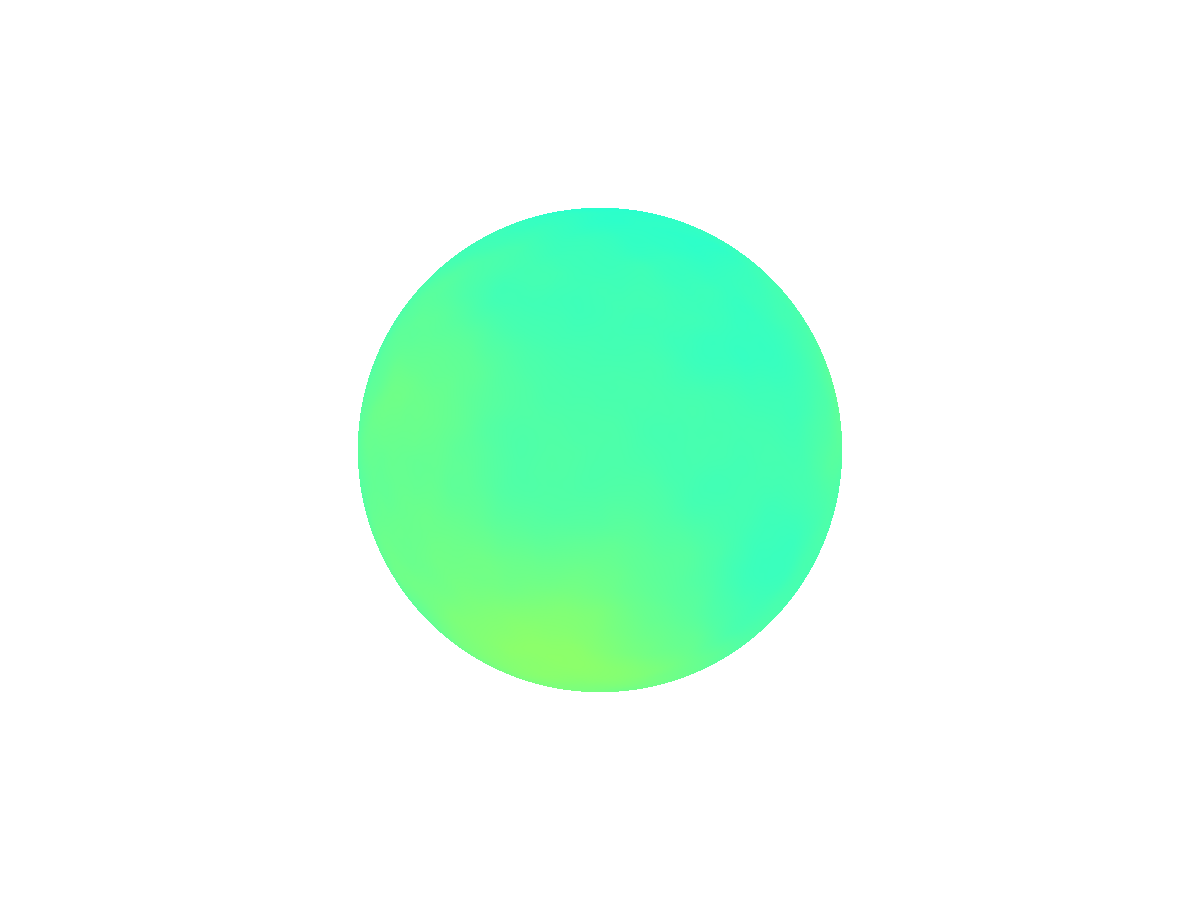}}
	\subfigure[$H=0.1$, $T=3$.]{
		\includegraphics[width=0.3\textwidth,trim={30em 15em 30em 15em},clip]{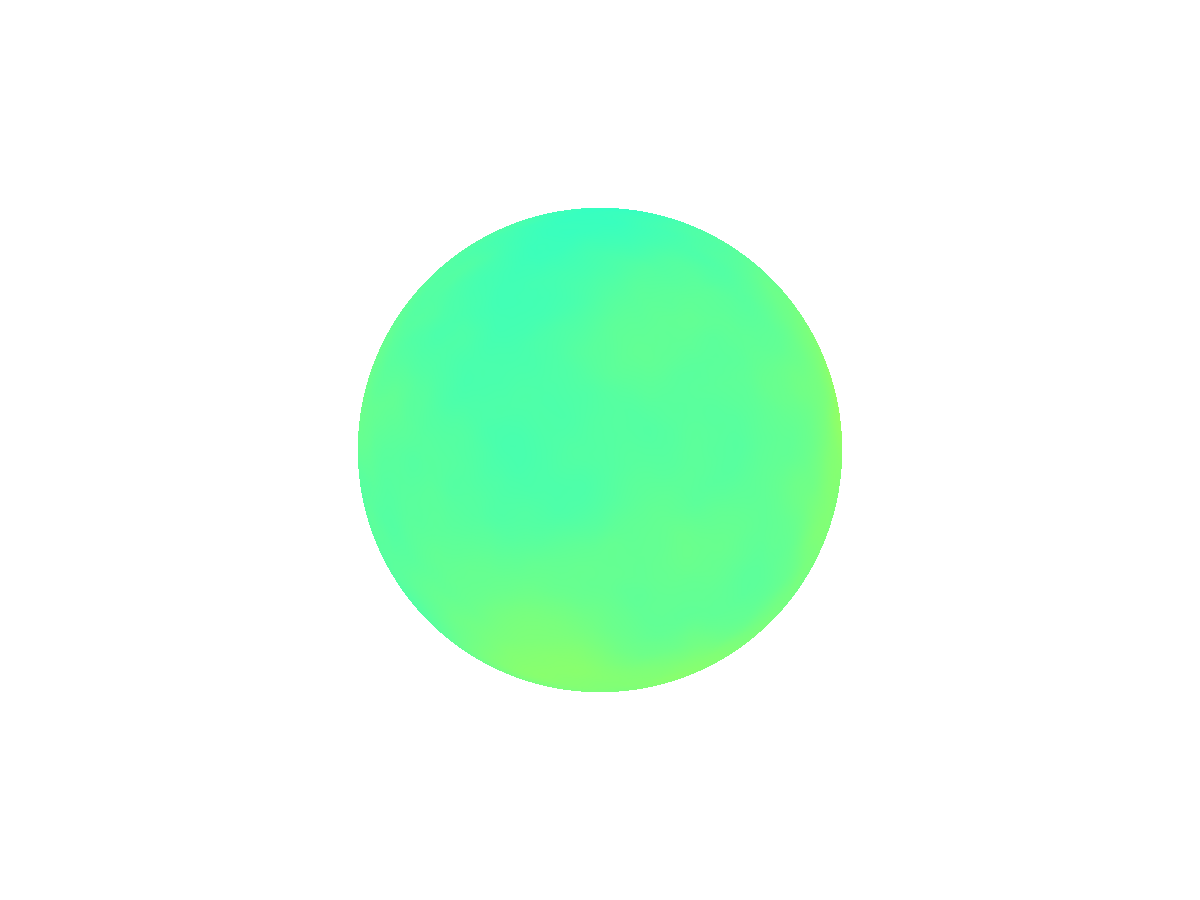}}\\
	\subfigure[$H=0.5$, $T=1$.]{
		\includegraphics[width=0.3\textwidth,trim={30em 15em 30em 15em},clip]{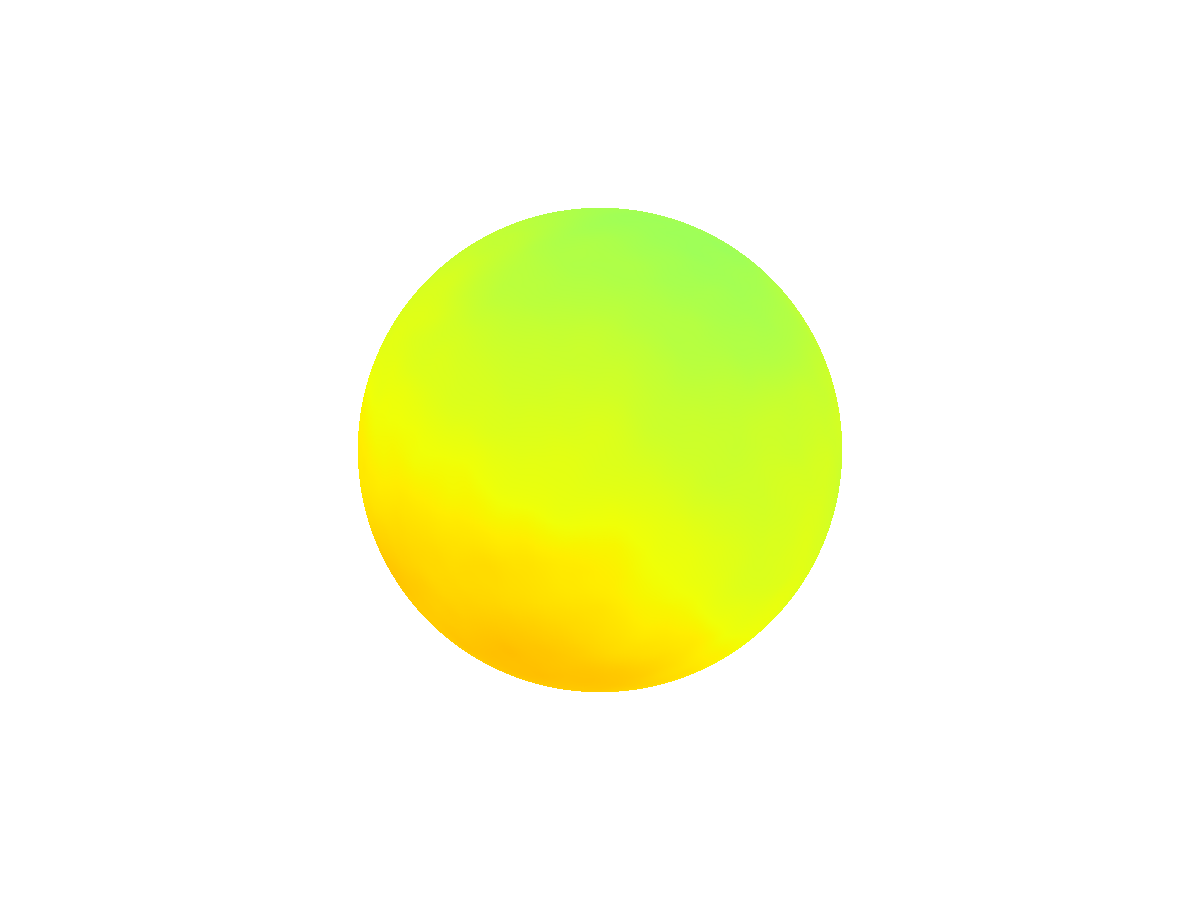}}
	\subfigure[$H=0.5$, $T=2$.]{
		\includegraphics[width=0.3\textwidth,trim={30em 15em 30em 15em},clip]{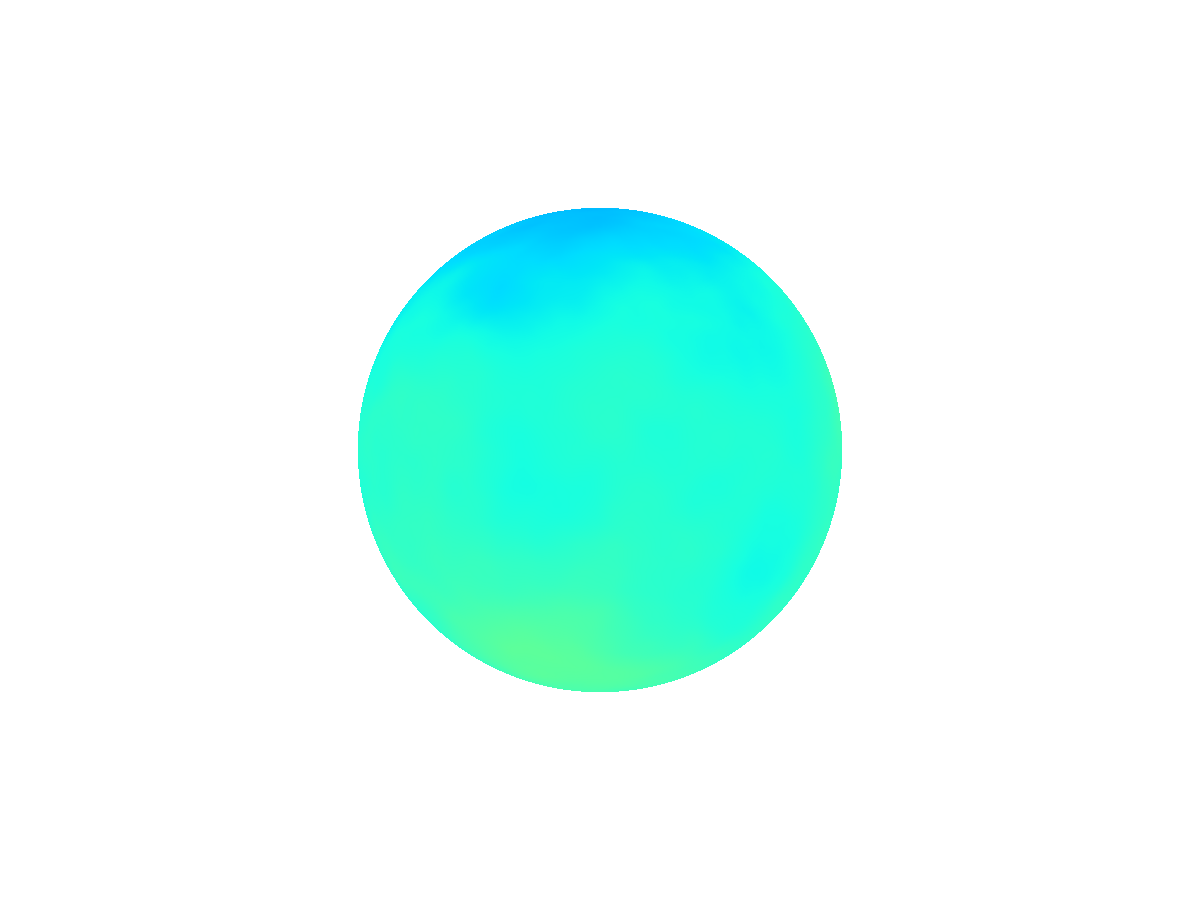}}
	\subfigure[$H=0.5$, $T=3$.]{
		\includegraphics[width=0.3\textwidth,trim={30em 15em 30em 15em},clip]{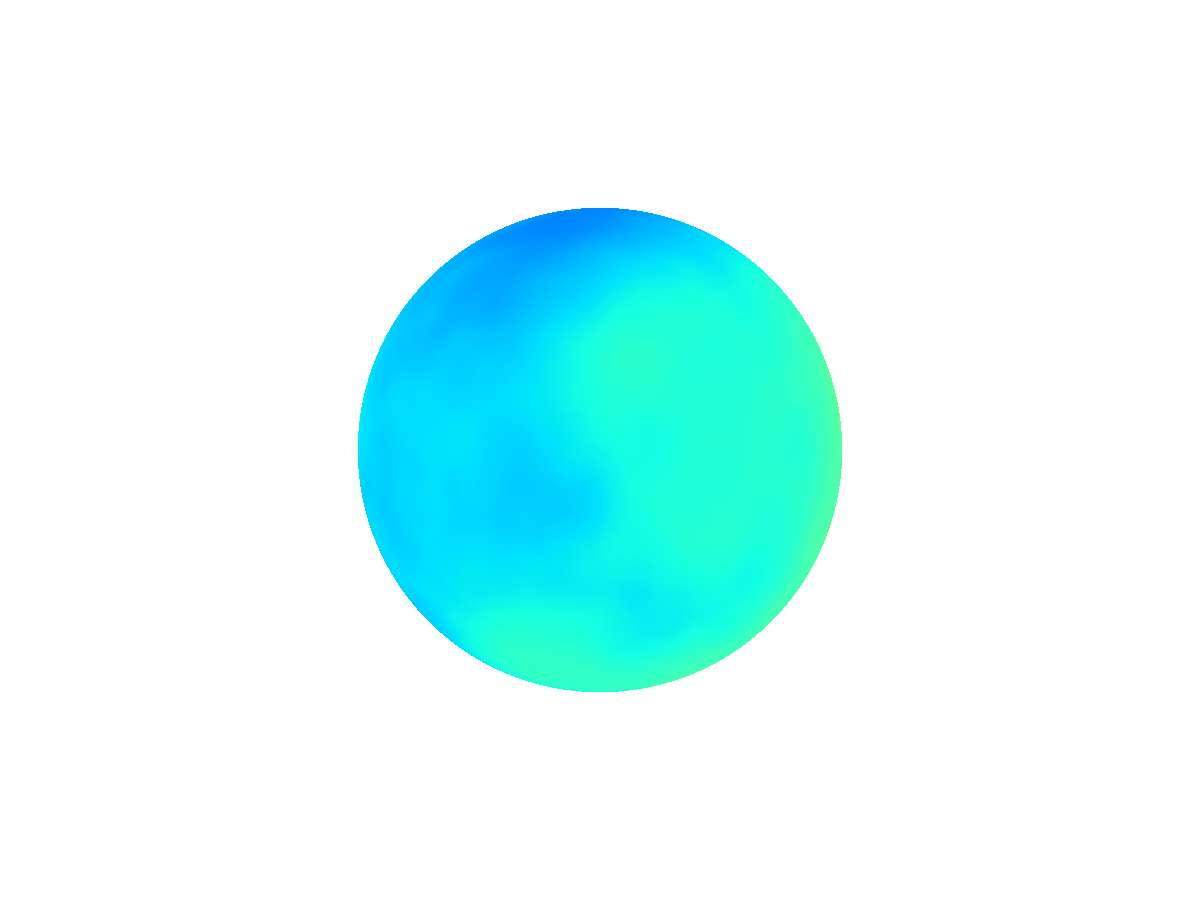}}\\
	\subfigure[$H=0.9$, $T=1$.]{
		\includegraphics[width=0.3\textwidth,trim={30em 15em 30em 15em},clip]{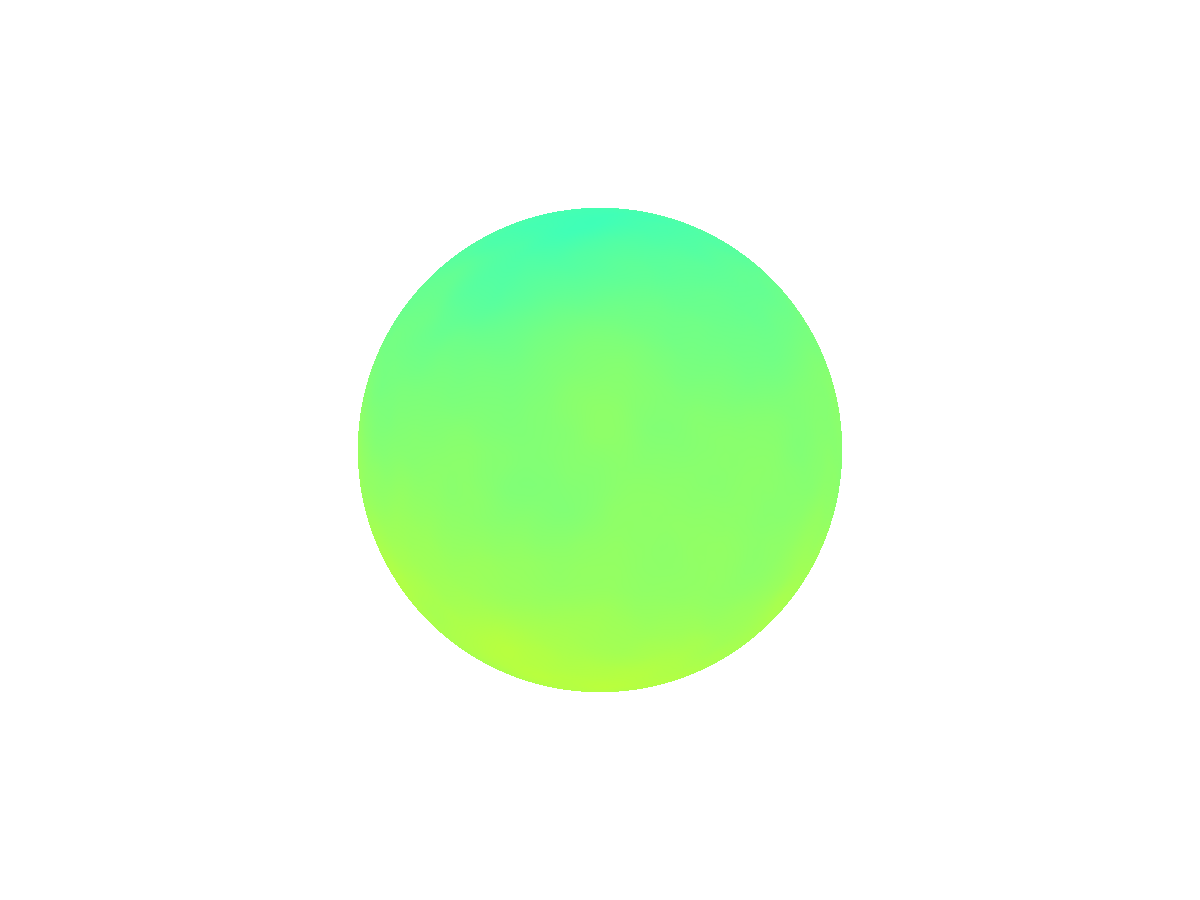}}
	\subfigure[$H=0.9$, $T=2$.]{
		\includegraphics[width=0.3\textwidth,trim={30em 15em 30em 15em},clip]{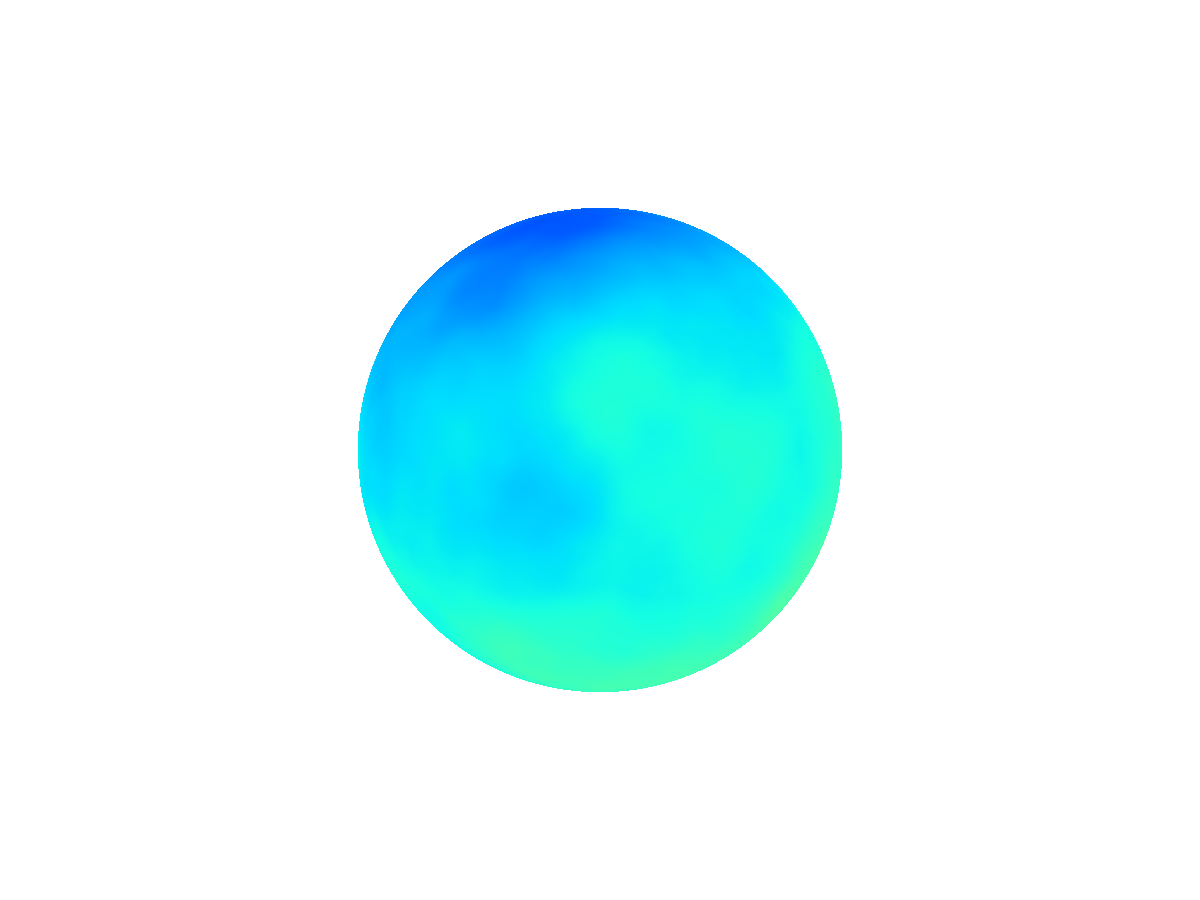}}
	\subfigure[$H=0.9$, $T=3$.]{
		\includegraphics[width=0.3\textwidth,trim={30em 15em 30em 15em},clip]{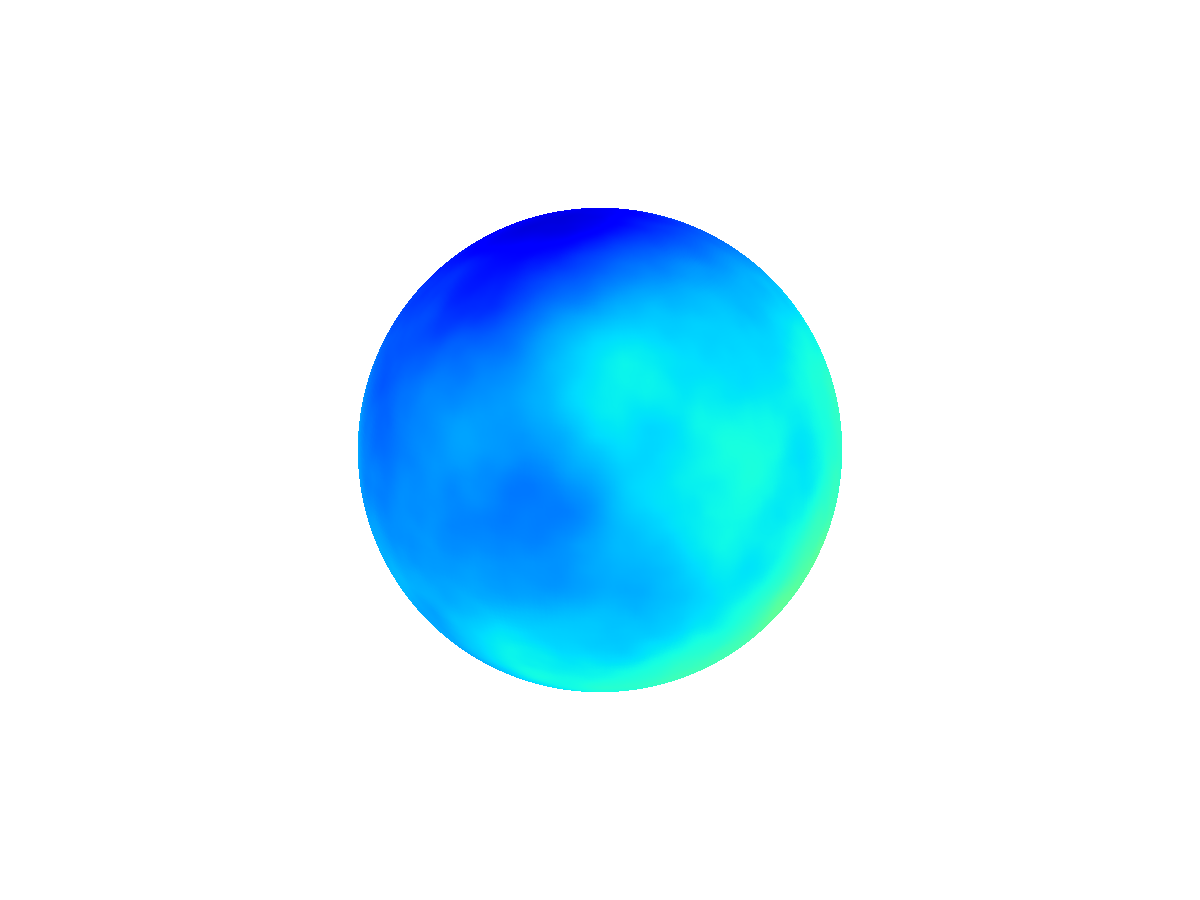}}
	\caption{Samples of $Q$-fBm for $H=0.1, 0.5, 0.9$ at time $T= 1, 2, 3$.\label{fig:sample_QfBm}}
\end{figure}

This article is organized as follows: In Section~\ref{sec:GRF_realfBm}, we shortly introduce the necessary background on real-valued fractional Brownian motion and Gaussian random fields on the unit sphere~$\IS^2$, to then define $Q$-fractional Brownian motion on~$\IS^2$ and analyze its space-time regularity in Section~\ref{sec:QfBm}. Section~\ref{sec:QfBm_Sd} contains the generalization of the results to $d$-dimensional spheres. The second part of the paper in Section~\ref{sec:simulation_QfBm} introduces a fully discrete approximation by a spectral method in space and circulant embedding or conditionalized random midpoint displacement in time. Strong and almost sure errors are analyzed and the performance and accuracy is shown numerically. The code that was used to generate the samples and numerical examples is available at~\cite{LM24_code}.

\section{Real-valued stochastic processes and spherical Gaussian random fields} \label{sec:GRF_realfBm}

$Q$-fractional Brownian motion on the sphere is a space-time stochastic process, which is constructed based on properties of a spatial Gaussian random field on the sphere and real-valued fractional Brownian motions. In this section, we first introduce the temporal processes and the spatial fields separately with their properties as basis for the $Q$-fractional Brownian motion in the next section.

Let us consider stochastic processes on the probability space $(\Omega, \cA, \IP)$ and on the finite time interval $\IT = [0,T]$. We recall that a \emph{real-valued fractional Brownian motion} (fBm) $\beta^H$ with Hurst parameter $H \in (0,1)$ is a continuous Gaussian process with mean zero and covariance
\begin{equation*}
	\phi_H(s, t) = \E\left[\beta^H(t)\beta^H(s)\right] = \frac{1}{2}(t^{2H} + s^{2H} - \lvert t-s \rvert^{2H}).
\end{equation*}
This process is H\"older continuous of order $\ga \in (0,H)$, which we abbreviate by $H^-$-H\"older continuous or $\beta^H \in C^{H^-}(\IT)$, in what follows.
It generalizes Brownian motion, which we recover for $H=1/2$.

We next consider properties of spatial processes or random fields on the sphere. We follow closely the introduction in \cite{langIsotropicGaussianRandom2015} and denote the \emph{unit sphere} by
\[
	{\IS^{2}} = \{x \in \R^3 | x_1^2 + x_{2}^{2} + x_{3}^{2} = 1\}
\]
and equip it with the \emph{geodesic distance}, defined for all $x, y\in {\IS^{2}}$ by
$d_{\IS^2}(x, y) = \arccos(\langle x, y \rangle_{\R^3})$.

Let $L^2({\IS^{2}})$ be the space of all real-valued square-integrable functions on ${\IS^{2}}$ and use the real-valued spherical harmonic functions $(Y_{\ell,m}, \ell \in \N_0, m = -\ell,\ldots,\ell)$ as orthonormal basis. 
A centered $L^2(\IS^2)$-valued \emph{isotropic Gaussian random field} (GRF)~$Z$ on~$\IS^2$ is given by the basis expansion, or Karhunen--Lo\`eve expansion,
\begin{equation}
	\label{eqn:KaL_spatial}
	Z = \sum_{\ell=0}^{\infty}\sum_{m=-\ell}^{\ell} \sqrt{A_\ell} \, z_{\ell, m} \, Y_{\ell , m},
\end{equation}
where $(A_\ell, {\ell\in\N_0})$, $A_{\ell} \geq 0$ for all $\ell\in\N_0$, is called the \emph{angular power spectrum} and $(z_{\ell , m}, \ell \in \N_0, m=-\ell,\ldots,\ell)$ is a sequence of independent, real-valued standard normally distributed 
random variables, as shown in \cite[Corollary~2.5]{langIsotropicGaussianRandom2015}.
The expansion~\eqref{eqn:KaL_spatial} converges in $L^2(\Omega \times {\IS^{2}})$ and for all $x\in {\IS^{2}}$ in $L^2(\Omega)$ \cite[Theorem~5.13]{MP11}.
The results for the real-valued spherical harmonics follow from the complex-valued expansion by \cite[Lemma~5.1]{langIsotropicGaussianRandom2015}.

The \emph{covariance kernel} of~$Z$ is given by
\begin{equation*}
	\phi_Q(x, y) = \E\left[Z(x)Z(y)\right] = \sum_{\ell=0}^{\infty} \sum_{m=-\ell}^{\ell} A_{\ell}Y_{\ell , m}(x)Y_{\ell , m}(y),
\end{equation*}
and the corresponding nonnegative and self-adjoint covariance operator~$Q$ is characterized by its eigendecomposition
\begin{equation*}
	Q Y_{\ell,m} = A_\ell Y_{\ell,m}
\end{equation*}
with finite trace $\trace Q = \sum_{\ell=0}^\infty (2\ell+1) A_\ell$, since $Z$ is an $L^2(\IS^2)$-valued Gaussian random variable.

While there exists a generalized theory that holds for $Q$ with infinite trace, the corresponding random fields would be of lower regularity than $L^2({\IS^{2}})$.
Instead, we are interested in a higher regularity, namely Hölder regularity, and therefore assume a scale of summability conditions on the angular power spectrum of $Q$, as given in~\cite{langIsotropicGaussianRandom2015}.
\begin{assumption}
\label{assump:summability}
Assume that the angular power spectrum $(A_\ell, \ell\in\N_0)$ of the covariance operator~$Q$ satisfies for some $\eta > 0$ that
$\sum_{\ell=0}^{\infty} A_\ell \ell^{1+\eta} < \infty$.
\end{assumption}
Under this assumption, $Z$ has a continuous modification that is in~$C^{(\eta/2)^-}(\IS^2)$, i.e., there exists a $C^{(\eta/2)^-}(\IS^{2})$-valued random field~$Y$ such that $\IP(Z(x) = Y(x)) = 1$ for all $x \in \IS^2$, as shown by~\cite{langIsotropicGaussianRandom2015}.
We use, for $\eta > 2$, the standard extension of Hölder spaces to orders greater than 1.

\section{\texorpdfstring{$Q$}{Q}-fractional Brownian motion on the sphere}\label{sec:QfBm}

Combining the temporal properties of real-valued fBm and spatial properties of isotropic GRFs on~$\IS^2$, we are now ready to define $Q$-fractional Brownian motion on~$L^2(\IS^2)$ following~\cite{ABOW18, greckschQfractionalBrownianMotion2009}.

\begin{definition}
	\label{def:Q_fbm}
	An $L^2(\IS^2)$-valued continuous Gaussian process $(B_{Q}^{H}(t))_{t\in\IT}$ with Hurst parameter $H \in (0, 1)$ is called a \emph{(standard) isotropic $Q$-fractional Brownian motion} ($Q$-fBm), if there exists an operator $Q$ satisfying Assumption~\ref{assump:summability}, such that for all $u, v \in L^2(\IS^2)$ and $s, t \in \IT$, $\E [\langle B_{Q}^{H}(t), u \rangle_{L^2(\IS^2)} ] = 0$ and
		\begin{equation*}
		\E\left[\langle B_{Q}^{H}(t), u \rangle_{L^2(\IS^2)} \langle B_{Q}^{H}(s), v\rangle_{L^2(\IS^2)}\right] 
		= \phi_H(t,s) \langle Qu, v \rangle_{L^2(\IS^2)}.
		\end{equation*}
\end{definition}

By the definition, we see that $B_Q^H$ is centered and the covariance splits into the temporal properties of real-valued fBm and the spatial description of isotropic GRFs on~$\IS^2$.
This becomes even more evident when citing existence and uniqueness of $Q$-fBm and its series expansion from~\cite{duncanFractionalBrownianMotion2002,greckschParabolicStochasticDifferential1999,greckschQfractionalBrownianMotion2009}.

\begin{theorem}
	\label{thm:existence}
	Let $Q$ satisfy Assumption~\ref{assump:summability} and $H \in (0, 1)$. Then, $Q$-fBm exists with basis expansion 
	\begin{equation*}
		B_{Q}^{H}(t) = \sum_{\ell=0}^{\infty}\sum_{m=-\ell}^{\ell} \sqrt{A_\ell}\beta_{\ell, m}^H(t)Y_{\ell , m},
	\end{equation*}
	where $(\beta_{\ell , m}^H, {\ell \in \N_0, m=-\ell,\ldots,\ell})$ is a sequence of independent real-valued fBms with Hurst parameter~$H$.
	Furthermore, $B_{Q}^{H} \in C^{H^-}(\IT; L^2(\IS^2))$.
\end{theorem}

We remark that \cite{duncanFractionalBrownianMotion2002} and \cite{greckschParabolicStochasticDifferential1999} only state the existence for $H>1/2$ but the existence proof of~\cite{duncanFractionalBrownianMotion2002} extends to all $H \in (0,1)$ since $Q$-fBm is a Gaussian process. 
Thus, the Kolmogorov--Chentsov theorem is still applicable for $H\leq 1/2$, since for all $n \in \N, s, t \in \IT$
\begin{equation*}
	\E\left[\lVert B_{Q}^{H}(t) - B_{Q}^{H}(s) \rVert_{L^2(\IS^2)}^{2n}\right] \leq \lvert t-s \rvert^{2nH}C_n(\trace Q)^n
\end{equation*}
for some constant~$C_n$ by \cite[Proposition~2.19]{dapratoStochasticEquationsInfinite2014}. Choosing $n$ such that $2nH > 1$, this allows us to apply \cite[Theorem~4.23]{kallenbergFoundationsModernProbability2021}.

We note that the series expansion in Theorem~\ref{thm:existence} matches for $H= 1/2$ the expansion of an isotropic $Q$-Wiener process on~$\IS^2$ as introduced in~\cite{langIsotropicGaussianRandom2015}.

Having considered $Q$-fBm so far as $L^2(\IS^2)$-valued, i.e., function-valued over~$\IS^2$, we are next interested in the spatial properties and in $Q$-fBm as a space-time process. For this, we first observe that
\begin{equation}\label{eq:KL_QfBm_x}
	B_{Q}^{H}(t, x) = \sum_{\ell=0}^{\infty} \sum_{m=-\ell}^{\ell} \sqrt{A_{\ell}} \beta_{\ell, m}^H(t)Y_{\ell, m}(x)
\end{equation}
is an isotropic Gaussian random field for fixed~$t$ (see~\cite{langIsotropicGaussianRandom2015}) that converges in $L^2(\gO;\R)$ pointwise in~$x$ by a version of the Peter--Weyl theorem, see~\cite[Theorem~5.13]{MP11}. It follows then that $B_{Q}^{H}$ is a Gaussian process on $\IT \times \IS^2$ since the linear combination $\sum_{k=1}^{n} \alpha_k B_{Q}^{H}(t_k, x_k)$ is Gaussian for any coefficients $(\alpha_k, k=1, \dots, n)$ and $((t_k,x_k), k=1,\dots, n)$ given the independent Gaussian processes~$\beta_{\ell , m}^H$.

We compute, as in \cite{ABOW18}, the covariance kernel~$k$ from~\eqref{eq:KL_QfBm_x}
\begin{align*}
	k(t, x, s, y) 
	& = \E\left[ B_{Q}^{H}(t, x) B_{Q}^{H}(s, y)\right]
	= \sum_{\ell=0}^{\infty} \sum_{m=-\ell}^{\ell} A_{\ell} \E\left[\beta_{\ell, m}^H(t)\beta_{\ell, m}^H(s)\right]Y_{\ell, m}(x)   Y_{\ell, m}(y) \\
	& = \phi_H(t, s) \phi_Q(x, y)
\end{align*}
with $\phi_H$ and $\phi_Q$ given in Section~\ref{sec:GRF_realfBm}.

Let us denote by $C^{\alpha, \beta}(\IT\times \IS^{2})$ the subspace of functions $f \in C^{\min\{\alpha,\beta\}}(\IT\times \IS^{2})$ such that for all $x \in \IS^2$, $f(\cdot,x) \in C^\alpha(\IT)$, and for all $t \in \IT$, $f(t,\cdot) \in C^\beta(\IS^2)$. Note that we interpret $\min\{\alpha^-, \beta^-\}$ as $\min\{\alpha, \beta\}^-$. We are now ready to state our main result on the space-time regularity of~$B_Q^H$.
	
\begin{theorem}\label{thm:fBm_space_time_reg}
	Let $Q$ satisfy Assumption~\ref{assump:summability}, then $B_{Q}^{H}$ has a continuous modification on $\IT \times \IS^{2}$ which is in $C^{H^-, (\eta/2)^-}(\IT\times \IS^{2})$.
\end{theorem}

To prove this theorem, we will first need to prove joint continuity in space-time with non-optimal parameters. 
For this, consider the compact Riemannian manifold $M= \IT\times {\IS^{2}}$ of dimension 3 equipped with the (topological) product metric 
\[d_M((t, x), (s, y)) = \lvert t - s \rvert + d_{\IS^{2}}(x, y)\]
for all $(t, x), (s, y) \in \IT\times {\IS^{2}}$. By \cite[Remark~2]{kratschmerKolmogorovChentsovType2023}, the Kolmogorov--Chentsov Theorem~1.1 in \cite{kratschmerKolmogorovChentsovType2023} applied to a Gaussian process on~$M$ becomes:

\begin{theorem}
	\label{thm:KC_kratschmer_Gaussian}
	Let $Z$ be a centered Gaussian process indexed by~$M$.
	Assume there exist $C > 0$ and $\xi \le 1$ such that for all $(t, x), (s, y) \in M$,
	\begin{equation}
		\label{eqn:std_bound}
		\E\left[\lvert Z(t, x)-Z(s, y)\rvert^2\right]^{1/2} \leq C d_M((t, x), (s, y))^\xi.
	\end{equation}
	Then, $Z$ has a continuous modification on~$M$, which is in~$C^{\xi^-}(M)$.
\end{theorem}

The proof uses a standard argument to compute $p$-th moments of Gaussian random variables based on the variance. For completeness, we give the proof in Appendix~\ref{app:proofs_KC}.

Applying this theorem to~$B_{Q}^{H}$, we obtain the following result.
\begin{corollary}
	\label{cor:joint_continuous}
	Assume that $Q$ satisfies Assumption~\ref{assump:summability}.
	Then, $B_{Q}^{H}$ has a continuous modification on~$M$, which is in $C^{\min\{H,\eta/2\}^-}(M)$.
\end{corollary}
\begin{proof}
We start the proof by splitting
\begin{align*}
	& \E\left[\left(B_{Q}^{H}(t, x) - B_{Q}^{H}(s,y)\right)^2\right]^{1/2}\\
		& \quad \le \E\left[\left(B_{Q}^{H}(t, x) - B_{Q}^{H}(s,x)\right)^2\right]^{1/2} + \E\left[\left(B_{Q}^{H}(s, x) - B_{Q}^{H}(s,y)\right)^2\right]^{1/2}.
\end{align*}
The first term satisfies
\begin{align*}
	\E\left[\left(B_{Q}^{H}(t, x) - B_{Q}^{H}(s,x)\right)^2\right]
		= \left(\phi_\beta(t, t) + \phi_\beta(s, s) - 2\phi_\beta(t, s)\right)\phi_Q(x,x)
		\le C_Q \lvert t -s \rvert^{2H},
\end{align*}
where $C_Q = \phi_Q(x,x) < \infty$ is constant since $B_{Q}^{H}$ is isotropic.
The second term is the increment of a Gaussian random field with angular power spectrum $(\phi_\beta(s,s)A_\ell, {\ell \in \N_0})$, which by \cite[Lemma~4.3]{langIsotropicGaussianRandom2015} and $\phi_\beta(s,s) = s^{2H} \le T^{2H}$ is bounded by
\begin{align*}
	\E\left[\left(B_{Q}^{H}(s, x) - B_{Q}^{H}(s,y)\right)^2\right]
		\le C_\eta \phi_\beta(s,s) d_{\IS^{2}}(x,y)^{\min\{\eta,2\}}
		\le C_\eta T^{2H} d_{\IS^{2}}(x,y)^{\min\{\eta,2\}}.
\end{align*}
Setting $\zeta = \min\{H,\eta/2\} < 1$, we obtain, since $z^\zeta$ is concave, for some constants $\tilde{C}$ and~$C$
\begin{equation*}
	\E\left[\left(B_{Q}^{H}(t, x) - B_{Q}^{H}(s,y)\right)^2\right]^{1/2}
		\le \tilde{C} \left(\lvert t -s \rvert^{\zeta} + d_{\IS^{2}}(x,y)^\zeta\right)
		\le C d_M((t,x), (s, y))^\zeta,
\end{equation*}
and applying Theorem~\ref{thm:KC_kratschmer_Gaussian} finishes the proof.	
\end{proof}

Without loss of generality, we denote by $B_{Q}^{H}$ this unique continuous modification.
For now, we have found the best possible (joint) Hölder exponent if we take the underlying space to be $(M, d_M)$. The next lemma will be used in the proof of Theorem~\ref{thm:fBm_space_time_reg} to obtain the ideal exponent for space and time separately.

\begin{lemma}
	\label{lem:pointwise_hoelder_time}
	Assume that $Q$ satisfies Assumption~\ref{assump:summability}.
	Then, for all $t \in \IT$, $B_{Q}^{H}(t, \cdot)$ has an indistinguishable modification that is in~$C^{(\eta/2)^-}(\IS^2)$, and similarly, for all $x \in \IS^2$, $B_{Q}^{H}(\cdot, x)$ has an indistinguishable modification that is in~$C^{H^-}(\IT)$.
\end{lemma}
\begin{proof}
	Since two continuous modifications are indistinguishable, the first claim follows from \cite[Theorem~4.6]{langIsotropicGaussianRandom2015}.
	In the proof of Corollary~\ref{cor:joint_continuous}, we showed that
	\begin{equation*}
		\E\left[(B_{Q}^{H}(t, x) - B_{Q}^{H}(s,x))^2\right] \leq C_Q \lvert t -s \rvert^{2H}.
	\end{equation*}
	Combining this with bounds of the $p$-th moments of Gaussian distributions 
	as in the proof of Theorem~\ref{thm:KC_kratschmer_Gaussian} and applying \cite[Theorem~4.23]{kallenbergFoundationsModernProbability2021} yields the claim.
\end{proof}

\begin{remark}
	Given the continuity, we conclude that for a fixed~$x$, $B_{Q}^{H}(\cdot, x)$ is a rescaled real-valued fBm since it is a Gaussian process satisfying $\E[B_{Q}^{H}(t, x)] = 0$ and $\E[B_{Q}^{H}(t,x) B_{Q}^{H}(s,x)] = \phi_\beta(t, s) \phi_Q(x, x)$.
\end{remark}

Now we have all results at hand to prove our main result on the space-time regularity of~$B^H_Q$.

\begin{proof}[Proof of Theorem~\ref{thm:fBm_space_time_reg}]	
The theorem is now a direct consequence of Corollary~\ref{cor:joint_continuous} and Lemma~\ref{lem:pointwise_hoelder_time}. The only remaining concern is if the indistinguishable process in Lemma~\ref{lem:pointwise_hoelder_time} with the higher regularity than in Corollary~\ref{cor:joint_continuous} introduces intractable null sets depending on $t$ or~$x$, respectively. However, since $B_{Q}^{H}$ is space-time continuous, the union of these null sets is another null set. This is shown by considering the null set obtained on a dense subset of $\IT$ or~$\IS^2$, respectively, and exploiting continuity. Therefore, we obtain an indistinguishable modification that is in~$C^{H^-, (\eta/2)^-}(\IT\times \IS^{2})$.
\end{proof}

\section{$Q$-fractional Brownian motion on \texorpdfstring{${\IS^{d-1}}$}{the d-dimensional hypersphere}}\label{sec:QfBm_Sd}
Analogous results hold when considering the hypersphere ${\IS^{d-1}}$ in $\R^d$ instead of $\IS^2$. In the framework of \cite{Y83}, we denote the real-valued spherical harmonics on~$\IS^{d-1}$ by  $(S^{(d-1)}_{\ell , m}, \ell \in \N_0, m=1,\ldots,h(\ell,d))$ with $h(\ell, d) = (2\ell + d - 2)\cdot(\ell + d - 3)!/((d-2)!\ell!)$.

Let $(\beta_{\ell, m}^H, \ell \in \N_0, m=1,\ldots,h(\ell,d))$ be a sequence of independent real-valued fBms. Assuming $\sum_{\ell=0}^{\infty}h(\ell, d)A_{\ell} < \infty$, we obtain, combining the results on~$\IS^2$ in Section~\ref{sec:QfBm} with Karhunen--Lo\`eve expansions on~$\IS^{d-1}$ for isotropic GRFs from~\cite{langIsotropicGaussianRandom2015}, the expansion
\begin{equation}\label{eqn:KL_Sd}
	B_{Q}^{H}(t) 
	= \sum_{\ell=0}^{\infty} \sum_{m=1}^{h(\ell, d)} \sqrt{A_{\ell}} \beta_{\ell, m}^H(t)S^{(d-1)}_{\ell, m}.
\end{equation}
To substitute $\IS^2$ by $\IS^{d-1}$ in Section~\ref{sec:QfBm}, we only need to apply the corresponding results for~$\IS^{d-1}$ from~\cite{langIsotropicGaussianRandom2015}. For that the generalized version of Assumption~\ref{assump:summability} becomes 
\begin{assumption}
	\label{assump:summability_d}
	Assume that the angular power spectrum $(A_\ell, {\ell\in\N_0})$ of the covariance operator $Q$ in~${\IS^{d-1}}$ satisfies for some $\eta > 0$ that
	$\sum_{\ell=0}^{\infty} A_\ell \ell^{d-2+\eta} < \infty$.
\end{assumption}
Replacing $d_{\IS^{2}}$ by $d_{\IS^{d-1}}$ and applying \cite[Theorem~4.7]{langIsotropicGaussianRandom2015}, the regularity results in Theorem~\ref{thm:fBm_space_time_reg} extend to $Q$-fBm on~$\IS^{d-1}$, which we state for completeness in the following theorem.
	\begin{theorem}
	Let $Q$ satisfy Assumption~\ref{assump:summability_d}, then $B_{Q}^{H}$ has a continuous modification which is in $C^{H^-, (\eta/2)^-}(\IT\times \IS^{d-1})$.
	\end{theorem}

\section{Efficient simulation of \texorpdfstring{$Q$}{Q}-fractional Brownian motion}\label{sec:simulation_QfBm}

In the past sections, we have characterized the regularity properties of $Q$-fBm in terms of its parameters $Q$ and~$H$. 
From the opposite perspective, we can now construct $Q$-fBms with given regularity properties by prescribing $Q$ and $H$ through Theorem~\ref{thm:fBm_space_time_reg}.
To use the process in applications, we need to be able to simulate it.
This section constructs and analyzes an approximation to the expansion~\eqref{eq:KL_QfBm_x} by truncating it and simulating independent sample paths of real-valued fBms.

\subsection{Spectral approximation in space}
\label{subsec:spectral_approx}

We return here to $\IS^2$ and truncate the basis expansion~\eqref{eq:KL_QfBm_x} at the parameter $\kappa \in \N$ to obtain the finite sum
\begin{equation}\label{eq:QfBm_spectral_approx}
  B_{Q}^{H, \kappa}(t, x) = \sum_{\ell=0}^{\kappa}\sum_{m=-\ell}^{\ell} \sqrt{A_\ell}\beta_{\ell, m}^H(t)Y_{\ell , m}(x).
\end{equation}

This sum can be interpreted analogously to a discrete Fourier transform (DFT), where $Y_{\ell , m}$ are the basis functions instead of complex exponentials. 
In fact, it can be rephrased in terms of DFTs and there exist implementations of the so-called \emph{Spherical Harmonics Transform} based on fast Fourier transforms (FFTs).
Since these FFT-based implementations allow for fast evaluation of $Q$-fBm on the sphere, we used them to generate the visualizations in Figure~\ref{fig:sample_QfBm}.

A spatial convergence analysis of the spectral approximation has been performed in \cite[Propositions~5.2 \&~5.3]{langIsotropicGaussianRandom2015} for a time-independent GRF on~${\IS^{2}}$ in $L^2(\gO;L^2(\IS^2))$ and $L^p(\gO;L^2(\IS^2))$ as well as $\IP$-a.s.\ in \cite[Corollary~5.4]{langIsotropicGaussianRandom2015}.
For a fixed $t$, their proofs apply to our situation up to a constant factor of $t^{H} < T^H$, noting that $\E[\beta_{\ell , m}^H(t)^2] = t^{2H}$.
This yields immediately the following theorem.

\begin{theorem}\label{thm:conv_trunc_KLexp_GRF}
  Let the angular power spectrum $(A_\ell, \ell \in \N_0)$ of the covariance operator $Q$
  decay algebraically with order $\ga>2$, i.e.,
  there exist constants $C>0$ and $\ell_0 \in \N$ such that 
  $A_\ell \le C \cdot \ell^{-\ga}$ for all $\ell > \ell_0$. 
  Then, the sequence of approximations $(B_{Q}^{H,\kappa}, \gk \in \N)$ 
  converges to $B_{Q}^{H}$ in~$L^p(\gO;L^2(\IS^2))$ for any finite $p \ge 1$ uniformly in $\IT$, and the error is bounded by
    \begin{equation*}
     \sup_{t \in \IT}\|B_{Q}^{H}(t) - B_{Q}^{H, \kappa}(t)\|_{L^p(\gO;L^2(\IS^2))}
      \le \hat{C}_p \, T^{H}\,  \gk^{-(\ga-2)/2}
    \end{equation*}
  for $\gk \ge \ell_0$, where $\hat{C}_p$ depends on $p$, $C$, and~$\ga$.
  This implies $\IP$-a.s.\ convergence such that for all $\gb < (\ga-2)/2$,
	the error is asymptotically bounded by $\gk^{-\gb}$. I.e., there exists a random variable $\kappa_0(t)$ such that for all $\kappa > \kappa_0(t)$,
\begin{equation*}
	\|B_{Q}^{H}(t) - B_{Q}^{H, \kappa}(t)\|_{L^2(\IS^2)} \le \gk^{-\gb},
	\quad \IP\text{-a.s.}.
\end{equation*}
\end{theorem}

A similar result is obtained on~$\IS^{d-1}$ when truncating the expansion~\eqref{eqn:KL_Sd} instead. The rate of convergence in $L^p(\gO;L^2(\IS^d))$ becomes $\gk^{(\alpha-d+1)/2}$ and therefore, $\gb < (\alpha-d+1)/2$ in $\IP$-a.s.\ sense. This is proven in the same way as Proposition~5.2 and Theorem~5.3 in~\cite{langIsotropicGaussianRandom2015}, where we use additionally that $h(\ell,d) \le C \ell^{d-2}$ by Stirling's inequality.

\subsection{Simulation of real-valued fractional Brownian motion}

Computing the above spectral approximation requires the simulation of independent sample paths of fBm.
Since fBm does not have independent increments like Brownian motion does, different simulation methods are required.
This is a widely explored topic, cf., e.g., \cite{diekerSimulationFractionalBrownian2004, wongSimulationFractionalBrownian2024b} and references therein.
A widely used method is the circulant embedding method (cf., e.g., \cite{perrinFastExactSynthesis2002} and references therein). If we accept an approximating algorithm, the conditionalized random midpoint displacement method (\cite{norrosSimulationFractionalBrownian1999}) can simulate a sample path of length~$N$ in $\cO(N)$ time.
We describe both methods here and do a performance comparison.
Both algorithms simulate the correlated increments of fBm which need to be added up to obtain sample paths.

\subsubsection{Circulant embedding}\label{subsec:CE}

The principle of the circulant embedding (CE) method is the same as that of the well-known Cholesky method: multiplying a standard Gaussian vector by the square root of the desired covariance matrix.
However, it makes use of the structure of the covariance matrix by employing the FFT to multiply a vector by the matrix square root.

To explain the method in the context of fBm, let $\IT_N$ be an equidistant time grid with $0= t_0 < \cdots < t_N = T$ and step size~$h$ and denote by $\Delta \gb^H(t_j) = \gb^H(t_{j+1}) - \gb^H(t_j)$ the correlated but stationary increments of~$\gb^H$. Setting $\gamma(|j-k|h) = \gamma(|t_j-t_k|) = \E[\Delta \gb^H(t_j)\Delta \gb^H(t_k)]$, the covariance matrix $\Sigma = (\gamma(|t_j-t_k|))_{j,k=0}^{N-1}$ (marked in grey below) is a Toeplitz matrix which can be embedded into the \emph{circulant matrix}
\newcommand{\cbg}{\cellcolor{gray!10}}
\begin{equation*}
	C = \tiny \begin{pmatrix}
		& \cbg \gamma(0) & \cbg \gamma(h) & \cbg \dots & \cbg \gamma((N-1)h) & \gamma((N-2)h) & \gamma((N-3)h) & \dots & \gamma(h)\\
		& \cbg\gamma(h) & \cbg \gamma(0) & \cbg \dots & \cbg \gamma((N-2)h) & \gamma((N-1)h) & \gamma((N-2)h) & \dots & \gamma(2h)\\
		& \cbg \vdots & \cbg \vdots & \cbg \ddots & \cbg \vdots& \vdots & \vdots & \ddots & \vdots\\
		& \cbg\gamma((N-1)h) & \cbg \gamma((N-2)h) & \cbg \dots & \cbg \gamma(0) & \gamma(h) & \gamma(2h) & \dots & \gamma((N-2)h)\\
		& \gamma((N-2)h) &  \gamma((N-1)h) &  \dots & \gamma(h) & \gamma(0) & \gamma(h) & \dots & \gamma((N-3)h)\\
		& \gamma((N-3)h) & \gamma((N-2)h) & \dots & \gamma(2h) & \gamma(h) & \gamma(0) &\dots &\gamma((N-4)h)\\
		& \vdots & \vdots & \ddots & \vdots & \vdots & \vdots & \ddots & \vdots \\
		& \gamma(h) & \gamma(2h) & \dots & \gamma((N-2)h) & \gamma((N-3)h) & \gamma((N-4)h) &\dots & \gamma(0)
	\end{pmatrix}
\end{equation*}
that is again a covariance matrix, as shown in \cite{perrinFastExactSynthesis2002}.
The description of the algorithm below follows~\cite{diekerSimulationFractionalBrownian2004} and~\cite{perrinFastExactSynthesis2002}.

We observe that $C = U \Lambda U^\ast$ has an eigendecomposition, where $U$ consists of the eigenvectors which are the Fourier modes and $\gL$ is a diagonal matrix with eigenvalues $(\gl_k, k=0,\ldots,2N-3)$ computed exactly by a DFT of the first row of~$C$.
These eigenvalues only need to be precomputed once for the generation of an arbitrary number of sample paths.
The square root of~$C$ is then given by $C^{1/2} = U \Lambda^{1/2}U^\ast$ and $\cN(0,C)$-distributed random samples can be generated by computing $C^{1/2} W$ for $W\sim \cN(0,\Id_{2N-2})$ via two DFT and a matrix multiplication, which can be performed in $\cO((2N-2)\log(2N-2))$ time via FFT. To omit one FFT, one can compute the distribution of $U^\ast W$ as exploited, e.g., in \cite{diekerSimulationFractionalBrownian2004,LP11}. Alternatively, one chooses, as in \cite{perrinFastExactSynthesis2002}, $W = V_1 + i V_2$ for independent $V_1, V_2 \sim \cN(0,\Id_{2N-2})$ and computes $Z = U\Lambda^{1/2}W$ by a single FFT of the vector 
$(\sqrt{\lambda_k/(2N-2)} w_k, k=1, \dots, 2N-2)$.
Then, $Z$ contains independent $\cN(0,C)$-distributed random samples in the real and imaginary part. The first $N$ entries of the random vectors $\Re(Z)$ and $\Im(Z)$ are the increments of fBm sample paths.

\subsubsection{CRMD}
\label{subsubsec:CRMD}

\definecolor{mplorange}{HTML}{FF7F0E}
\definecolor{mplblue}{HTML}{1F77B4}
\definecolor{mplgreen}{HTML}{2CA02C}
\newcommand\newinc{black}
\newcommand\known{black}
\newcommand\depend{mplorange}
\newcommand\used{mplblue}
\begin{figure}
	\centering
  \begin{tikzpicture}[scale=0.35]
	\node at (15, 0) (node1) {{$X_{0,1}$}};
	\draw[{|[right]}-{|[left]}, \newinc, line width=1.5pt] (0,-1) -- (30,-1);
	\node at (7.5, -2) (node1) {{$X_{1,1}$}};
	\draw[{|[right]}-{|[left]}, \newinc, line width=1.5pt] (0,-3) -- (15,-3);
	\node at (22.5, -2) (node1) {{$X_{1,2}$}};
	\draw[{|[right]}-{|[left]}, \known, loosely dashed, line width=1.5pt] (15,-3) -- (30,-3);
	\node at (7.5/2, -4) (node1) {{$X_{2,1}$}};
	\draw[{|[right]}-{|[left]}, \newinc, line width=1.5pt] (0,-5) -- (7.5,-5);
	\node at (7.5+7.5/2, -4) (node1) {{$X_{2,2}$}};
	\draw[{|[right]}-{|[left]}, \known, loosely dashed, line width=1.5pt] (7.5,-5) -- (15,-5);
	\node at (15+7.5/2, -4) (node1) {{$X_{2,3}$}};
	\draw[{|[right]}-{|[left]}, \used, line width=1.5pt] (15,-5) -- (15+7.5,-5);
	\node at (15+7.5+7.5/2, -4) (node1) {{$X_{2,4}$}};
	\draw[{|[right]}-{|[left]}, \depend, loosely dashed, line width=1.5pt] (15+7.5,-5) -- (30,-5);
  
	\node at (7.5/4, -6) (node1) {{$X_{3,1}$}};
	\draw[{|[right]}-{|[left]}, \depend, line width=1.5pt] (0,-7) -- (7.5/2,-7);
	\node at (7.5/2+7.5/4, -6) (node1) {{$X_{3,2}$}};
	\draw[{|[right]}-{|[left]}, \depend, loosely dashed, line width=1.5pt] (7.5/2,-7) -- (15/2,-7);
	\node at (15/2+7.5/4, -6) (node1) {{$X_{3,3}$}};
	\draw[{|[right]}-{|[left]}, \used, line width=1.5pt] (15/2,-7) -- (15/2+7.5/2,-7);
	\node at (15/2+7.5/2+7.5/4, -6) (node1) {{$X_{3,4}$}};
	\draw[{|[right]}-{|[left]}, \used, loosely dashed, line width=1.5pt] (15/2+7.5/2,-7) -- (15,-7);

	\node at (18.75, -5) (x23) {};
	\node at (16.875, -7) (x35) {};
	\draw [-latex, bend left, line width=1.5pt] (x23) edge (x35);
  
	\node at (15/2+7.5/4, -7) (x33) {};
	\node at (15/2+7.5/2+7.5/4, -7) (x34) {};

	\draw [-latex, bend right, line width=1.5pt] (x34) edge (x35);

	\draw [-latex, bend right, line width=1.5pt] (x33) edge (x35);

	\node at (15+7.5/4, -6) (node1) {{$X_{3,5}$}};
	\draw[{|[right]}-{|[left]}, mplgreen, line width=2pt] (15,-7) -- (15+7.5/2,-7);
	\node at (15+7.5/2+7.5/4, -6) (node1) {{$X_{3,6}$}};
	\draw[{|[right]}-{|[left]}, \known, loosely dashed, line width=1.5pt] (15+7.5/2,-7) -- (15+15/2,-7);
	\node at (15+15/2+7.5/4, -6) (node1) {{$X_{3,7}$}};
	\draw[{|[right]}-{|[left]}, \newinc, line width=1.5pt] (15+15/2,-7) -- (15+15/2+7.5/2,-7);
	\node at (15+15/2+7.5/2+7.5/4, -6) (node1) {{$X_{3,8}$}};
	\draw[{|[right]}-{|[left]}, \known, loosely dashed, line width=1.5pt] (15+15/2+7.5/2,-7) -- (15+15,-7);
  \end{tikzpicture}
  \caption{Visual representation of the CRMD method ($\mu=2, \nu=1$). \textcolor{mplgreen}{$X_{3,5}$ is simulated} \textcolor{mplblue}{conditional on $X_{2,3}$, $X_{3,3}$, and $X_{3,4}$,} \textcolor{mplorange}{ignoring its dependence on $X_{2,4}$, $X_{3,1}$, and $X_{3,2}$.}}
  \label{fig:CRMD}
\end{figure}
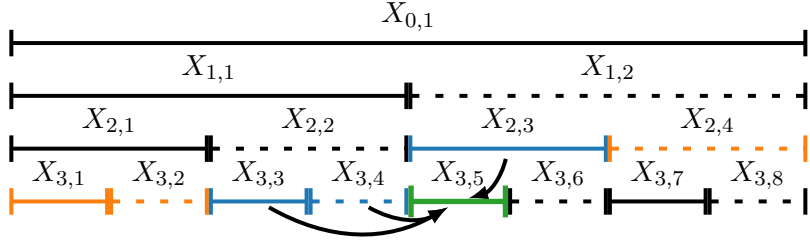

The conditionalized random midpoint displacement (CRMD) method is based on L\'evy's construction of Brownian motion using the Brownian bridge, and we introduce it following \cite{norrosSimulationFractionalBrownian1999}.

Similarly to Subsection~(\ref{subsec:CE}), we want to simulate a sample path of~$\beta^H$ on the equidistant time grid~$\IT_N$ with $T=1$, $N = 2^{n_0}$, and $n_0 \in \N$.
The initialization step is to simulate the increment $X_{0,1} = \beta^H(1) - \beta^H(0) \sim \cN(0, 1)$.
The discretization grid is then dyadically refined such that in the $n$-th step for $n \leq n_0$, we want to generate the increments $(X_{n, k} = \beta^H(k 2^{-n}) - \beta^H((k-1)2^{-n}), k=1,\ldots,2^n)$ given by the recursive relationship
$$ X_{n-1, k} = X_{n, 2k-1} + X_{n, 2k}, \qquad k=1, \dots, 2^{n-1}. $$
Assume now that we have simulated refinement step $n-1$ and $X_{n, 1}, \dots, X_{n, 2k-2}$.
In contrast to the Markovian structure of Brownian motion, where the next increment to be simulated, $X_{n, 2k-1}$, only depends on~$X_{n-1, k}$, the long-range dependence of fBm requires to condition on all previously known increments $M_{n,2k-1} = (X_{n-1, k}, \dots, X_{n-1, 2^{n-1}}, X_{n, 1}, \dots, X_{n, 2k-2})$. As computed in detail in~\cite{norrosSimulationFractionalBrownian1999}, $X_{n, 2k-1}$ is a Gaussian random variable with conditional mean $\E[X_{n, 2k-1}|M_{n,2k-1}]$
and variance $\Var[X_{n, 2k-1}|M_{n,2k-1}]$ that can be computed via the covariance matrix of the extended vector $(X_{n, 2k-1}, M_{n,2k-1})$. Once $X_{n, 2k-1}$ has been sampled, we compute $X_{n, 2k} = X_{n-1, k} - X_{n, 2k-1}$.
This procedure is exact but of computational complexity~$\cO(N^2)$.

To decrease the complexity to $\cO(N)$, we need to compute the conditional distributions more efficiently while accepting that the resulting samples are approximate sample paths. For that, let us reduce the maximal number of increments to condition on from the right and left to $\nu,\mu \in \N$ and set
$M^{\nu, \mu}_{n, 2k-1} = (X_{n-1, k}, \dots, X_{n-1, \min(k+\nu, 2^{n-1})}, X_{n, \max(2k-1-\mu, 1)}, \dots, X_{n, 2k-2})$.
We sample now $X_{n, 2k-1}$ based on the approximate conditional mean 
$\E[X_{n, 2k-1}|M^{\nu, \mu}_{n,2k-1}]$ and variance $\Var[X_{n, 2k-1}|M^{\nu, \mu}_{n,2k-1}]$. This is illustrated in Figure~\ref{fig:CRMD}.
Note that $\E[X_{n, 2k-1}|M^{\nu, \mu}_{n,2k-1}]$ can be computed as a dot product of a fixed vector $\mathbf{e}$ with $M^{\nu, \mu}_{n,2k-1}$ and $\Var[X_{n, 2k-1}|M^{\nu, \mu}_{n,2k-1}] = v$ is a scalar value that only needs to be rescaled.
Only for points near the boundary, up to $\nu \cdot \mu$ different $\mathbf{e}$ and $v$ are required, which can also be precomputed.
The vectors $\mathbf{e}$ are of size at most $\mu + \nu$, and therefore the computational complexity is reduced to $\cO((\mu+\nu)N)$, i.e., linear in~$N$. This makes CRMD asymptotically faster than CE.

Theoretical error bounds in $\nu$ and $\mu$ are to the best of our knowledge unknown. Dieker \cite{diekerSimulationFractionalBrownian2004} compares the covariance functions for different $\mu$ and~$\nu$ and performs statistical tests on the obtained sample distributions. To estimate the strong error numerically, we set $\nu = \lceil \mu/2 \rceil$ due to the symmetry in the algorithm, and simulate sample paths with different values of $\mu$ but using the same random numbers as input. 
Denoting by $\beta^{H, \mu}$ the approximation of fBm by CRMD, Figure~\ref{fig:conv_in_m} shows how the error $\sup_{t\in\IT_N} \lVert \beta^{H, \mu}(t) - \beta^H(t) \rVert_{L^2(\Omega)}$ decays as $\mu$ increases, based on $M=10^4$ Monte Carlo samples for varying $H$ with $N=512$.
The exact (reference) solution~$\beta^H$ is computed with $\mu = N$.
Table~\ref{tab:rates} shows in more detail the empirically obtained decay rates of this error for a larger number of values of $H$.
Note that we decided to estimate the rates in the range $\mu = s, \dots, 128, s = 10, 20, 50,$ in order to exclude the less regular behavior for $\mu < s$.
\begin{table}
\resizebox{\textwidth}{!}{
\begin{tabular}{lccccccccccccccc}
	H & 0.01 & 0.05 & 0.1 & 0.2 & 0.3 & 0.4 & 0.45 & 0.49 & 0.51 & 0.55 & 0.6 & 0.7 & 0.8 & 0.9 \\
	$r_H$, $s=10$ & 0.88 & 0.85 & 0.80 & 0.81 & 0.89 & 1.03 & 1.17 & 1.59 & 1.95 & 0.87 & 0.96 & 1.14 & 1.26 & 1.37 \\
	$r_H$, $s=20$ & 0.92 & 0.86 & 0.81 & 0.81 & 0.88 & 1.00 & 1.12 & 1.53 & 1.88 & 0.84 & 1.01 & 1.16 & 1.28 & 1.39 \\
	$r_H$, $s=50$ & 0.96 & 0.88 & 0.83 & 0.83 & 0.85 & 0.95 & 1.06 & 1.42 & 1.54 & 0.92 & 1.05 & 1.19 & 1.30 & 1.40 
\end{tabular}
}
\caption{\label{tab:rates}Empirical decay rates $r_H$ of the error $\sup_{t\in\IT_N} \lVert \beta^{H, \mu}(t) - \beta^H(t) \rVert_{L^2(\Omega)}$.}
\end{table}
Given this data, we estimate for $0 \ll \mu \ll N$ that the error decays with rates $r_H$ around $1$ that depend on $H$ and are bounded from below by $0.8$.
Therefore, we obtain the empirical error bounds
\begin{equation}\label{eq:conv_fBm_CRMD}
	\sup_{t \in \IT_N} \, \lVert \beta^{H, \mu}(t) - \beta^H(t) \rVert_{L^2(\Omega)} \le C \mu^{-r_H}.
\end{equation}
The constant $C$~depends on~$H$ but appears to be independent of~$N$, provided $N \gg \mu$.

The fully discrete approximation of the $Q$-fBm~\eqref{eq:KL_QfBm_x} on the sphere based on the spectral approximation~\eqref{eq:QfBm_spectral_approx} and CRMD is then given by
\begin{equation*}
	B_{Q}^{H, \kappa, \mu}(t,x) 
	= \sum_{\ell=0}^{\kappa}\sum_{m=-\ell}^{\ell} \sqrt{A_\ell} \, \beta_{\ell, m}^{H,\mu}(t) \, Y_{\ell , m}(x).
\end{equation*}
Its strong error can be split into
	\begin{align*}
	&\lVert B_{Q}^{H}(t) - B_{Q}^{H, \kappa, \mu}(t) \rVert_{L^2(\Omega; L^2({\IS^{2}}))}\\
	& \qquad  \leq \lVert B_{Q}^{H}(t) - B_{Q}^{H, \kappa}(t) \rVert_{L^2(\Omega; L^2({\IS^{2}}))} + \lVert B_{Q}^{H, \kappa}(t) - B_{Q}^{H, \kappa, \mu}(t) \rVert_{L^2(\Omega; L^2({\IS^{2}}))},
\end{align*}
where the first term is bounded by Theorem~\ref{thm:conv_trunc_KLexp_GRF}. The second term satisfies based on~\eqref{eq:conv_fBm_CRMD} 
\begin{equation*}
	\lVert B_{Q}^{H, \kappa}(t) - B_{Q}^{H, \kappa, \mu}(t) \rVert_{L^2(\Omega; L^2({\IS^{2}}))}
		\le C \sqrt{\trace Q} \, \mu^{-r_H}.
\end{equation*}
This allows us to conclude the analysis of CRMD with the following corollary.

\begin{corollary}
	Under the assumptions of Theorem~\ref{thm:conv_trunc_KLexp_GRF} and~\eqref{eq:conv_fBm_CRMD} with $\mu \ll N$, the strong error of the fully discrete spectral and CRMD approximation of~\eqref{eq:KL_QfBm_x} is bounded by
	\begin{equation*}
		\sup_{t \in \IT_N} \lVert B_{Q}^{H}(t) - B_{Q}^{H, \kappa, \mu}(t) \rVert_{L^2(\Omega; L^2({\IS^{2}}))} 
			\le C (\kappa^{-(\alpha -2 )/2} + \sqrt{\trace Q} \, \mu^{-r_H}).
x	\end{equation*}
\end{corollary}

\begin{figure}[tb]
	\includegraphics[width=\textwidth]{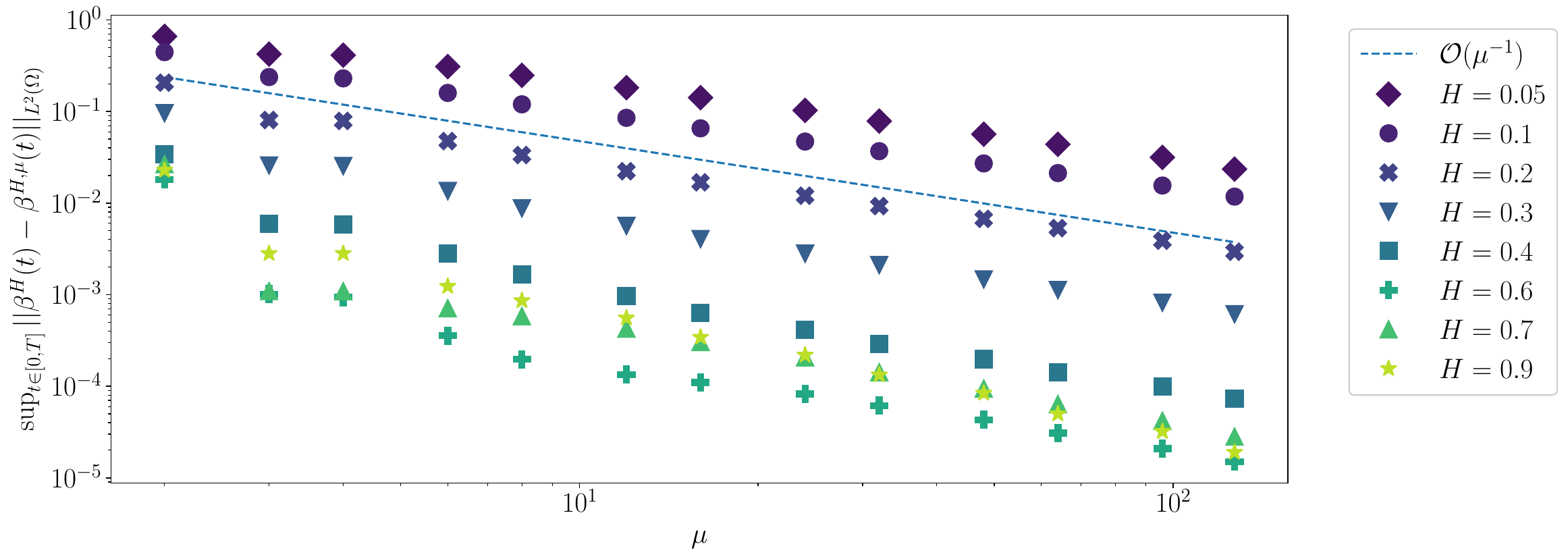}
		\caption{\label{fig:conv_in_m}Empirical decay of the error of CRMD.}
\end{figure}
\begin{figure}[tb]
	\includegraphics[width=\textwidth]{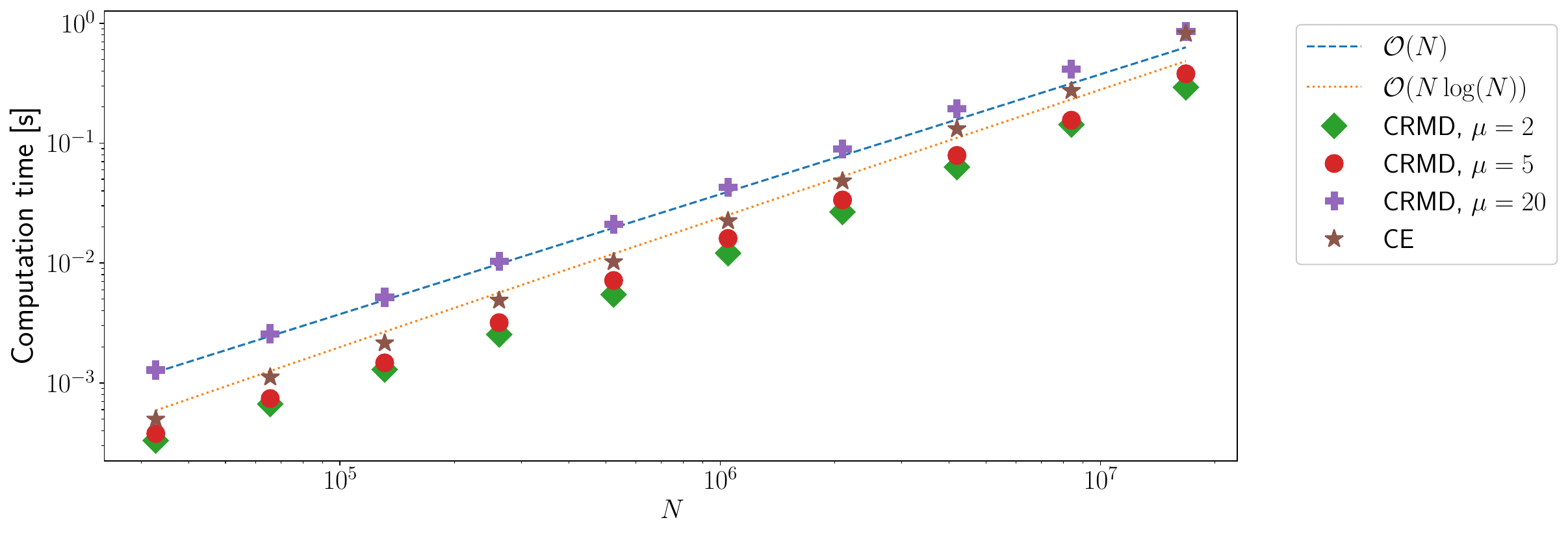}
		\caption{\label{fig:perf}Performance comparison of CE and CRMD.}
\end{figure}

\subsubsection{Comparison of computational performance}

We have seen that the computational costs of CE behave asymptotically as $\cO(N \log N)$ while CRMD performs with linear complexity. In this section, we compare their performance for relevant choices of~$N$ based on our implementation in Julia. We check when the constants hidden by the $\cO$-notation matter compared to the extra $\log N$ factor in the asymptotics.

Figure~\ref{fig:perf} shows the computation time required by both methods for the simulation of sample paths of varying length~$N$, ranging from $2^{15}$ to $2^{24} \approx 1.68 \cdot 10^7$, with $H=0.8$.
Note that both methods perform optimally when $N$ is a power of~2 and allow for the precomputation of certain steps that depend only on~$N$ (and $\mu$ for CRMD) and do not need to be repeated for every new sample path. The time taken by this was excluded in our analysis here.
The computation was performed single-threaded on an Intel\textsuperscript{\textregistered} Core\texttrademark \ i5-1245U system with 16 GB of RAM.

A slowdown of CRMD for higher values of $\mu$ is observed, since the computational cost per increment increases due to the required evaluation of wider dot products.
For small values of $\mu$, the CRMD method is faster than the CE method, while for larger $\mu$, CRMD becomes slower than CE.

We remark that CRMD can be implemented to only require memory for the $N$ floating point values representing the resulting sample path and $\cO(\mu^3)$ values to store the vectors ${\mathbf{e}}$ and scalars~$v$ that are used to compute the conditional mean and variance.
CE, on the other hand, requires about $6N$ floating point values to be stored since the $2N$ (real-valued) eigenvalues of the covariance matrix $C$ need to be stored and the output is a complex vector of length $2N$.

Considering the results in Figure~\ref{fig:perf}, the choice of method is a trade-off between accuracy and computational performance.
If low accuracy is sufficient, better performance can be obtained by using CRMD with small~$\mu$.
On the other hand, if higher accuracy is required, CE is the method of choice.
Our tests show that, on our system, it is not advisable to use CRMD with $\mu = 20$ since we can obtain distributionally exact samples with the same computational costs using the CE method.
Dieker \cite{diekerSimulationFractionalBrownian2004} performed a similar comparison of computational cost, albeit including the precomputation steps, and reported similar relative costs. However, we note a significant speedup of all computations.

Finally, we note that CE can be expressed quite simply in terms of Fourier transforms, for which highly optimized library implementations of FFT are available.
CRMD, on the other hand, is a significantly more complex algorithm. 
Hence, from an implementation and usability perspective, CE is preferred.

\bibliographystyle{habbrv}
\bibliography{fBM_spheres_full}

\appendix

\section{Proof of Theorem~\ref{thm:KC_kratschmer_Gaussian}}\label{app:proofs_KC}

This theorem arises from \cite[Theorem~1.1]{kratschmerKolmogorovChentsovType2023} by the standard argument on finiteness of all moments of a Gaussian distribution.
By assumption, we know that $Z(t, x) - Z(s, y) \sim \cN(0, \sigma^2)$ for some $\sigma$. 
Then, the standard deviation is $\E\left[\lvert Z(t, x)-Z(s, y)\rvert^2\right]^{1/2} = \sigma$ since the Gaussian is centered.
We further know that 
\begin{equation*}
	\E\left[\lvert Z(t, x)-Z(s, y)\rvert^p\right] = \E\left[\lvert \sigma U\rvert^p\right] = \sigma^p\E\left[\lvert U\rvert^p\right] = C(p) \sigma^p
\end{equation*}
for some $U \sim \cN(0, 1)$ and all $p \in \N$.
Now, by assumption~\eqref{eqn:std_bound}, $\sigma \leq C d_M(x, y)^\xi$ and hence
\begin{equation*}
	\E\left[\lvert Z(t, x)-Z(s, y)\rvert^p\right] \leq C(p) C^p d_M(x, y)^{p\xi} = \tilde{C}_p d_M(x, y)^{p\xi}.
\end{equation*}
Choose $p > \frac{d}{\xi}$ arbitrary, and  $q = \xi p$. 
This yields 
\begin{equation*}
	\E\left[\lvert Z(t, x)-Z(s, y)\rvert^p\right] \leq \tilde{C}_p d_M(x, y)^q,
\end{equation*}
and the assumptions of \cite[Theorem~1.1]{kratschmerKolmogorovChentsovType2023} are satisfied, where our $\tilde{C}_p$ is their $M$.
Note that our space $M$ is a compact Riemannian manifold of dimension 3, cf.\ \cite[Remark~2]{kratschmerKolmogorovChentsovType2023}. Thus, we obtain the desired bounds and a modification that is $\beta$-Hölder continuous for all $\beta \in (0, q/p - d/p) = (0, \xi - d/p)$. Letting $p$ tend to infinity, we obtain that there is a modification that is in $C^{\xi^-}(M)$.
\qed

\end{document}